\documentclass[11pt,a4paper,reqno]{amsart}
\usepackage{float}
\usepackage{amsfonts,amsmath,graphics,epsfig,latexsym,times}
\usepackage[T1]{fontenc}
\newif\ifpix \pixtrue

\usepackage[colorlinks=true]{hyperref}
\hypersetup{ pdfstartview=FitH }

\numberwithin{equation}{section}

\numberwithin{equation}{section}
\newcommand{\vol}{\mathrm{vol}}
\newcommand{\ghyp}{g_\mathrm{hyp}}
\newcommand{\spinc}{\ensuremath{\mathrm{Spin}^c}}
  
\def\ZZ{{\mathbb{Z}}}     \def\CC{{\mathbb{C}}}
\def\RR{{\mathbb{R}}}

\def\cV{{\mathcal{V}}}
\def\cM{{\mathcal{M}}}

\renewcommand{\epsilon}{\varepsilon}

\newcommand{\gs}{\mathfrak{s}}

\newcommand{\del}{\partial}

\newcommand{\Xbar}{{\overline{X}}}

\newtheorem{prop}[subsection]{Proposition}

\newtheorem{theointro}{Theorem}

\newtheorem{corintro}[theointro]{Corollary}

\theoremstyle{definition}

\newtheorem{rmk}[subsection]{Remark}

\newtheorem*{rmk*}{Remark}
\newtheorem*{rmks*}{Remarks}
\newtheorem{rmks}[subsection]{Remarks}

\newenvironment{remark*}{\begin{rmk*} --- \normalfont} { \end{rmk*} }
\newenvironment{remarks*}{\begin{rmks*} \begin{enumerate} \normalfont}
{\end{enumerate} \end{rmks*} }

\title[A remark on the Herzlich volume of ACH Einstein
manifolds]{A remark on the Herzlich volume of asymptotically complex
  hyperbolic  Einstein
  manifolds}

\author{Yann Rollin}
\address{Yann Rollin, Imperial College, London, UK}
\email{rollin@imperial.ac.uk}
\thanks{The author is  supported by a University Research
  Fellowship of the Royal Society}


\begin{document}
{\Huge \sc \bf\maketitle}
\begin{abstract}
We observe inequalities involving  the Herzlich 
volume of  a $4$-dimensional asymptotically complex hyperbolic
 Einstein manifold and its  Euler
characteristic provided the metrics is either K\"ahler or selfdual. In the selfdual case we have to assume
furthermore that
 the Kronheimer-Mrowka invariant is non vanishing.
\end{abstract}

\section{Statement of results}
In this short note, we observe inequalities involving the Herzlich 
volume of  $4$-dimensional \emph{asymptotically complex hyperbolic Einstein} manifold  $(X,g)$ and the Euler
characteristic $\chi(X)$,  provided the metric is either
\emph{K\"ahler} or \emph{selfdual}. In the selfdual case we have to
assume  that  the Kronheimer-Mrowka invariant is moreover non vanishing.

The acronym \emph{ACH}  shall be used as a shorthand in
 the rest of this paper for \emph{asymptotically
 complex hyperbolic}. Our main result is stated next, however the reader yearning for
basic definitions  may want to read Section~\ref{sec:def} first.  

Given   an orientable $4$-dimensional ACH
Einstein manifold $(X,g)$, Herzlich introduces in~\cite{Her07} an
invariant of the metric $\cV(g)$, closely
related to the \emph{renormalized volume}  $V(g)$. An essential
feature of
this invariant proved by Herzlich,  is  the Gauss-Bonnet type formula
\begin{equation}
  \label{eq:gb}
\chi(X)=\frac 1{8\pi^2}\int_X\left ( |W_g|^2 -\frac 1{24}s_g^2\right)
\vol_g +\frac 1{4\pi^2}\cV(g), 
\end{equation}
where $W_g$ denotes the Weyl curvature of $g$ and $s_g$ the scalar curvature.

\begin{theointro}
\label{theo:main}
  Let $\Xbar$ be a $4$-dimensional oriented manifold  with boundary,
 such that its 
  interior $X$ carries an  ACH Einstein
  metric $g$.

If $g$ is K\"ahler with respect to a complex structure
  compatible with the 
orientation of $\Xbar$, we have
\begin{equation}
\label{eq:involk}
{4\pi^2}\chi(X)\geq \cV(g)
\end{equation}
where $\cV(g)$ is the  Herzlich volume of $g$.
In addition  equality holds if and only if the metric  is complex hyperbolic.

If    $g$ is selfdual and
 $(\Xbar,\xi)$ admits a monopole class, where $\xi$ be the (positive cooriented) contact structure on $\del\Xbar$ induced by the conformal
infinity of $g$, then 
\begin{equation}
\label{eq:involsd}
{4\pi^2}\chi(X)\leq \cV(g)
\end{equation}
and equality holds if and only if the metric  is complex hyperbolic.
\end{theointro}
\begin{rmk}
  The above result shows a clear disjunction between the selfdual
  Einstein   metrics~\eqref{eq:involsd} and the K\"ahler Einstein
  metrics~\eqref{eq:involk}. 
One can compare Theorem~\ref{theo:main} to a similar and yet dramatically
different result 
due to Anderson in the real case: given a $4$-dimensional orientable \emph{asymptotically
real hyperbolic} (ARH) Einstein manifolds $(X,g)$, we have 
$$V(g)\leq
\frac {4\pi^2}3\chi(X),$$
 where $V(g)$ is the renormalized volume of
$g$. Moreover, equality holds if and only if the 
metric is real hyperbolic. The
proof follows immediately  from a
Gauss-Bonnet formula similar to~\eqref{eq:gb} in the real setting
(cf.~\cite{And01} for more details). However, the inequality holds for
every ARH Einstein metric and no disjunction phenomenon appears in
this case. 
  \end{rmk}
 \subsection*{ Acknowledgments }
The author wishes to thank Marc Herzlich for some useful comments.
\section{Definitions}
\label{sec:def}
  Let $\Xbar^4$ be an oriented  manifold with oriented 
boundary $Y^3$ and interior
  $X\subset \Xbar$ 
  endowed with an ACH Einstein
  metric $g$.

Recall that   a \emph{cooriented contact distribution}  $\xi\subset TY$ is
given as the kernel 
  of a globally defined $1$-form $\eta$, defined up to multiplication
  by a positive
  function, such that $\eta\wedge d\eta$ 
  is a volume form. The contact structure is \emph{positive} if this volume
  form is compatible with the orientation of $Y$. 

A \emph{strictly pseudoconvex} CR structure $J$ is given by an almost complex
structure $J$ defined along a positive cooriented  contact structure
$\xi=\ker \eta$ such that $\gamma(\cdot,\cdot)=d\eta(\cdot,J\cdot)$
is a   Hermitian  metric defined along $\xi$.  
With the above notations,
 $\xi=\ker\eta$ is refered to as the (positive cooriented) contact
 structure  \emph{induced} by $J$.

Identify a collar neighborhood of $Y$ in $\Xbar$ with $[0,T)\times Y$, with
coordinate $u$ on the first factor.
 On this collar
neighborhood we have a model metric
$$ g_0 = \frac{du^2+\eta^2}{u^2} + \frac{\gamma}{u} $$
We say that a metric Riemannian $g$ on $X$ is ACH, with \emph{conformal
infinity} the strictly pseudoconvex CR structure $J$,
if near the boundary one has
\begin{equation}
 g = g_0 + \kappa ,\label{eq:10}
\end{equation}
where $\kappa$ is a symmetric 2-tensor, such that $|\nabla^k\kappa| =
O(u^\delta)$ for some $\delta>0$ and all $k\geq 0$ (here, all the
norms and derivatives are taken with respect to the metric $g_0$).

Kronheimer and Mrowka introduced in \cite{KroMro97} a suitable version
of Seiberg-Witten theory for compact oriented manifolds with  positive
cooriented contact
boundary. Therefore the theory applies to $(\Xbar,\xi)$.
The contact structure $\xi$ induces a canonical \spinc-structure on $Y$
denoted  $\gs_\xi$.
Recall that the space $\spinc(\Xbar,\xi)$ is the set of equivalence
classes of pairs $(\gs,h)$, where $\gs$ is a \spinc-structure on
$\Xbar$ and $h:\gs|_{Y}\to\gs_\xi$ is an isomorphism.
Then, the Kronheimer-Mrowka invariant is given by a map
$$
sw_{\Xbar,\xi} :\spinc(\Xbar,\xi) \to \ZZ
$$
defined up to an overall sign, and obtained by ``counting'' the
solutions of Seiberg-Witten equations. If $sw_{\Xbar,\xi}(\gs,h)\neq
0$ for some $(\gs,h)\in\spinc(\Xbar,\xi)$, we say that $(\gs,h)$ is a \emph{monopole class}.

\section{Examples of K\"ahler and selfdual ACH Einstein metrics}

The only known examples of selfdual ACH Einstein metric at the moment consist
of
\begin{enumerate}
\item  \label{ex:h} the complex hyperbolic plane,
and, more generally,  ACH
manifolds obtained by taking suitable quotients of the complex
hyperbolic plane by a group of isometries,
\item \label{ex:ks}the Calderbank-Singer examples~\cite{CalSin04}, defined in a
  neighborhood of the
  exceptional divisor of the minimal resolution of certain isolated
  orbifold singularities.
\end{enumerate}
whereas the essential source of ACH K\"ahler-Einstein metrics in
addition to hyperbolic examples \eqref{ex:h}  is
 provided by the Cheng-Yau metrics on strictly pseudoconvex
domains of $\CC^2$ (and more generally of $\CC^n$). 
Other examples due to Biquard   are obtained by
deformation  
of the complex hyperbolic
metric \cite{Biq00,Biq05} in the selfdual Einstein and the
K\"ahler-Einstein cases. 
Finally  further ACH K\"ahler-Einstein examples are  obtained
by drilling wormholes in their conformal infinity~\cite{BiqRolX06}.

\begin{rmk}
\label{rmk:fill}
In each example given above, the boundary contact
structure $\xi$ induced by the conformal infinity of the metric  is
symplectically  
fillable. Therefore $(\Xbar,\xi)$ admits a monopole class by
\cite[Theorem 1.1]{KroMro97}. It follows that  Theorem \ref{theo:main}
applies.
\end{rmk}

We shall give now an application  of  Theorem~\ref{theo:main}, 
in the case of the ball. Consider the moduli space  
$$
\cM=\{\mbox{ACH Einstein metrics on $B^4$}\}/\{\mbox{diffeomorphisms
  of $B^4$}\}
$$
Let $\cM^\mathrm{K}\subset\cM$ be the subspace of K\"ahler metrics and
$\cM^\mathrm{SD}$ the subspace of selfdual metric.  Notice that
$\cM^\mathrm{K}$ contains all the Cheng-Yau metrics on strictly
pseudoconvex domains of $\CC^2$ diffeomorphic to a ball. 
  Let $\cM^\bullet_{{std}}$ be the 
  component of the moduli space $ \cM^\bullet$ consisting of
  metrics with conformal infinity inducing the standard 
  contact structure $\xi_{std}$ on the sphere $S^3$, up to
  diffeomorphism. Since contact structure are stable under small
  deformation, it follows that $\cM^\bullet_{{std}}$ is a
  neighborhood of $\ghyp$ in $\cM^\bullet$.

\begin{rmk} The contact structure $\xi$ induced by the conformal infinity
of every metric in $\cM^\mathrm{K}$ is automatically symplectically
fillable by the K\"ahler form. It follows by a theorem of Eliashberg
that $\xi$ is tight. Since $\xi_{std}$ is the only tight contact
structure on $S^3$, it turns out that
$\cM^\mathrm{K}_{std}=\cM^\mathrm{K}$.
By contrast, it is an open question whether
  $\cM^\mathrm{SD}_{{std}}$ is a  proper subset of $\cM^\mathrm{SD}$. 
\end{rmk}

\begin{rmk}
Small deformations of the conformal infinity of $\ghyp$ are in $1:1$
correspondence with small ACH Einstein deformations of the
metric~\cite{Biq00}. Moreover the moduli spaces $\cM^\mathrm{K}$ and
$\cM^\mathrm{SD}$ intersect transversely at $\ghyp$~\cite{Biq05}.
Corollary~\ref{cor:main}  says that in some sense
$\cV:\cM_{std}\to \RR$ admits a saddle point at $\ghyp$.
Thus,  there are certain
ACH Einstein deformations $g$ of the complex hyperbolic metric, which are
neither K\"ahler nor selfdual and such that $\cV(g)=\cV(\ghyp)$. More generally,
it would be interesting to understand precisely, how deformations of the
conformal infinity act on the Herzlich volume of the corresponding ACH
Einstein metric.
\end{rmk}

\begin{corintro} 
\label{cor:main}
The Herzlich volume $\cV:\cM^\mathrm{SD}_{{std}}\to \RR$
  is bounded below, and admits 
  a unique minimum at the complex hyperbolic metric
  $\cV(\ghyp)=4\pi^2$, whereas $\cV:\cM^\mathrm{K}\to \RR$
  is bounded above, and admits 
  a unique maximum at the complex hyperbolic metric.
\end{corintro}
A similar version of the Corollary  holds 
manifolds obtained by hyperbolic quotient as in examples~\eqref{ex:h}.


\section{Proofs}
 \begin{proof}[Proof of Corollary~\ref{cor:main}]
Theorem~\ref{theo:main} applies automatically to every metric in
$\cM^\mathrm{K}$.

Notice that the standard symplectic form on $B^4$ is a filling of
the contact structure $\xi_{std}$.
It follows, according to Remark~\ref{rmk:fill}, that
Theorem~\ref{theo:main} applies to every metric in 
  $\cM^\mathrm{SD}_{{std}}$  as well, and the Corollary is proved.
 \end{proof}

 \begin{proof}[Proof of Theorem~\ref{theo:main}]
If the metric is K\"ahler, we have the pointwise identity
${s^2_g} = 24|W^+_g|^2$. 
Using the Gauss-Bonnet formula~\eqref{eq:gb}, we have
 $$0\geq \int_X \left ( \frac{s^2_g}{24} -|W_g|^2\right )\vol_g=2\cV(g)-8\pi^2\chi(X)$$
with equality if and only if the metric is selfdual and inequality
\eqref{eq:involk} is proved.

The inequality ~\eqref{eq:involsd} is way less trivial.
However the most difficult part of the proof is carried out
in~\cite{Rol04}. The reader may want to consult  a short note in
English as well, stating the result without the technicalities~\cite{Rol02CR}. 

Given any  oriented manifold $\Xbar^4$,  endowed
with 
an ACH Einstein metric $g$, let $\xi$ be the cooriented positive contact
structure induced by the conformal infinity. Assume that 
there exists a monopole class
class $(\gs,h)\in\spinc(\Xbar,\xi)$.  Then one
can construct a solution of Seiberg-Witten equations w.r.t $(\gs,h)$
and  $g$,   with a suitable
decay at the boundary (cf. \cite[Th\'eor\`eme 2]{Rol04}).
Once a solution of Seiberg-Witten equations is obtained, one can prove the following result (cf.  Proposition 29 and the end of the
proof of Corollaire 31 in \cite{Rol04}), 
\begin{prop}
 Let $\Xbar^4$ be an oriented manifold with boundary endowed with
an ACH Einstein metric and let $\xi$ be the contact structure induced
by the conformal infinity on the boundary.
If  $(\Xbar,\xi)$ admits a monopole
class then
$$0\leq \int_X \left (\frac{s^2_g}{24} -|W^+_g|^2\right )\vol_g$$
with equality if and only if the metric is K\"ahler-Einstein.
\end{prop}

In particular, if the metric is selfdual, we have 
$$0\leq \int_X \left (\frac{s^2_g}{24} -|W_g|^2\right )\vol_g=2\cV(g)-8\pi^2\chi(X).$$
Moreover equality holds if and only if the metric is K\"ahler.
 This proves the inequality \eqref{eq:involsd}. 

Finally inequalities  \eqref{eq:involk} and  \eqref{eq:involsd}  are
saturated  if and only if the metric $g$ is both selfdual and
K\"ahler-Einstein, and therefore, complex hyperbolic.
\end{proof}

\bibliographystyle{plain}
\bibliography{$HOME/tex/biblio,$HOME/tex/rollin}

\end{document}